\documentclass{article}
\usepackage[utf8]{inputenc}
\usepackage[]{graphicx}
\usepackage{amsthm,amsmath,amssymb,tikz,tikz-cd,amsmath,mathrsfs,mathtools,multicol,dirtytalk,tabularx,xy,lipsum,url,enumitem,cmll}

\usepackage{enumitem}

\usepackage{blkarray}
\usepackage{tabularray}
\usepackage[colorlinks=true,linkcolor=blue,citecolor=blue]{hyperref}

\numberwithin{equation}{section}

\theoremstyle{plain}
\newtheorem{theorem}{Theorem}[section]
\newtheorem{lemma}[theorem]{Lemma}
\newtheorem{proposition}[theorem]{Proposition}
\newtheorem{corollary}[theorem]{Corollary}

\theoremstyle{definition}
\newtheorem{definition}[theorem]{Definition}
\newtheorem{remark}[theorem]{Remark}
\newtheorem{example}[theorem]{Example}

\setlist[itemize]{noitemsep}

\title{\textbf{Canonical completion and duality for cylindric ortholattices and cylindric Boolean algebras}}
\author{Joseph McDonald\footnote{University of Alberta, Department of Philosophy, Edmonton, T6G 2E7, Canada.}\footnote{Email: jsmcdon1@ualberta.ca\\This research was supported under the CGS-D SSHRC grant no.~767-2022-1514.}}
\date{}

\begin{document}

\maketitle
\begin{abstract}
 In this note, we investigate the algebraic and topological representation theory of cylindric ortholattices and cylindric Boolean algebras. The first contribution demonstrates that cylindric ortholattices are closed under canonical completions. By equipping a spectral topology to the dual space associated with the canonical completion, we then establish a dual equivalence between the category of cylindric ortholattices and a certain subcategory of the category of spectral spaces. This work builds on the completion and duality results obtained by Harding, McDonald, and Peinado in the setting of monadic ortholattices combined with the duality results obtained by McDonald and Yamamoto in the setting of general ortholattices. By working with the duality theory for Boolean algebras established by Bezhanishvili and Holliday, we then obtain completion and duality results for cylindric Boolean algebras. A key aspect of our duality results is that they are constructive in the sense that they obtain in Zermelo-Fraenkel set theory independently of the Axiom of Choice.   
       
\par
\vspace{.2cm}
\end{abstract}

\section{Introduction}

Cylindric Boolean algebras were introduced by Henkin and Tarski \cite{henkin} and by Henkin, Monk, and Tarski \cite{tarski1,tarski2} as algebraic models of the classical predicate calculus with identity. Closely related to the cylindric Boolean algebras are the polyadic and monadic Boolean algebras studied by Halmos \cite{halmos, halmos1}, which form algebraic models of the classical single-variable predicate calculus and the full predicate calculus without identity, respectively. For any cardinal $\kappa$, a $\kappa$\emph{-dimensional cylindric Boolean algebra} is a Boolean algebra $B$ equipped with a family $(\exists_{\alpha})_{\alpha\in\kappa}$ of pairwise commuting quantifiers (i.e., closure operators whose closed elements form Boolean sub-algebras of $B$) and a family $(d_{\alpha,\beta})_{\alpha,\beta\in\kappa}$ of constants satisfying certain conditions. The standard construction of a $\kappa$-dimensional cylindric Boolean algebra involves starting with a set $X$ and then considering the concrete (powerset) Boolean algebra of the function space $X^{\kappa}$. For each $\alpha\in\kappa$, its \emph{cylindrification} $\exists_{\alpha}$ of a subset $Y\subseteq X^{\kappa}$ is given by the collection of choice-functions $\sigma\in X^{\kappa}$ that agree with a choice-function in $Y$ except (possibly) at the coordinate $\alpha$. The \emph{diagonal} $d_{\alpha,\beta}$ is the collection of all choice-functions $\sigma\in X^{\kappa}$ satisfying the equation $\sigma(\alpha)=\sigma(\beta)$. Recently, Harding \cite{harding} introduced cylindric ortholattices and studied them for their connections to subfactors within the setting of von Neumann algebras. A key aspect distinguishing cylindric ortholattices from cylindric Boolean algebras is that the former do not in general satisfy the distributive identities.

In this note, we first prove that cylindric ortholattices are closed under canonical completions. This extends the canonical completion results obtained by Harding, McDonald, and Peinado \cite{harding3} within the setting of monadic ortholattices. The completion of a cylindric ortholattice $A$ arises by forming an associated dual space $X_A$ that is a cylindric orthoframe. This is a set equipped with an orthogonality relation, an indexed family of commuting reflexive, transitive binary relations, as well as an indexed family of distinguished subsets satisfying certain conditions. By working with the proper filters of $A$, the corresponding completion is obtained by embedding $A$ into the complete cylindric ortholattice of biorthogonally closed subsets of $X_A$. Based on the obtained canonical completion result, our main result contributes to the recent research program of producing (spatial) topological dualities for lattice-based algebras independently of the Axiom of Choice. Research along these lines was initiated within the setting of Boolean algebras by Bezhanishvili and Holliday \cite{bezhanishvili}, which can be viewed as arising in part via Stone's representation of bounded distributive lattices via the compact open subsets of a spectral space \cite{stone2}, the Vietoris hyperspace construction \cite{vietoris}, and Tarski's observation \cite{tarski1, tarski2} that the regular open subsets of any topological space form a Boolean algebra. These methods were subsequently applied in the setting of ortholattices by McDonald and Yamamoto \cite{mcdonald}, De Vries algebras by Massas \cite{mas}, modal algebras by Bezhanishvili, Dmitrieva, \emph{et.~al.} \cite{dmitrieva}, as well as Heyting algebras, implicative lattices, and monotone bounded lattice expansions by Hartonas \cite{hartonas1, hartonas2}. 

The primary aim of this note is to apply these methods to the setting of cylindric ortholattices and cylindric Boolean algebras. We first introduce certain spectral spaces, which we call \emph{cylindric upper Vietoris orthospaces}. These are relational extensions of the upper Vietoris orthospaces (UVO-spaces) introduced by McDonald and Yamamoto \cite{mcdonald} combined with the monadic orthospaces studied by Harding, McDonald, and Peinado \cite{harding3}. It is shown that every cylindric ortholattice $A$ is isomorphic to the algebra $\mathcal{A}_0(\mathcal{S}_0(A))$ of compact open biorthogonally closed subsets of the spectrum $\mathcal{S}_0(A)$ of proper filters of $A$ equipped with the spectral topology. Unlike the representation and duality results for ortholattices in \cite{goldblatt1, bimbo} and for monadic ortholattices in \cite{harding3}, ours avoids the use of Alexander's Subbase Theorem, or any other choice-dependent principle, in proving that $\mathcal{S}_0(A)$ forms a compact topological space. We then prove that every cylindric UVO-space $X$ is relationally homeomorphic to the spectrum $\mathcal{S}_0(\mathcal{A}_0(X))$ of proper filters of the algebra $\mathcal{A}_0(X)$ of compact open biorthogonally closed subsets of $X$. Finally, with the introduction of suitable spectral frame morphisms, which we call \emph{cylindric UVO-maps}, we show that the category $\mathcal{COL}$ of cylindric ortholattices is dually equivalent to the category $\mathcal{CUVO}$ of cylindric UVO-spaces. We then introduce relational extensions of the UV-spaces introduced by Bezhanishvili and Holliday \cite{bezhanishvili}, which we call \emph{cylindric UV-spaces}. These are certain spectral spaces that come equipped with a family of commuting equivalence relations as well as a family of distinguished subsets satisfying similar conditions to cylindric UVO-spaces. We show that every cylindric Boolean algebra $A$ is isomorphic to the algebra $\mathcal{G}_0(\mathcal{F}_0(A))$ of compact open regular open subsets of the spectrum $\mathcal{F}_0(A)$ of proper filters of $A$. We then show that every cylindric UV-space $X$ is relationally homeomorphic to the spectrum $\mathcal{F}_0(\mathcal{G}_0(X))$ of proper filters of the algebra $\mathcal{G}_0(X)$ of compact open regular open subsets of $X$. With the introduction of suitable spectral frame morphisms, which we call \emph{cylindric UV-maps}, we show that the category $\mathcal{CBA}$ of cylindric Boolean algebras is dually equivalent to the category $\mathcal{CUV}$ of cylindric UV-spaces. In doing so, we provide a topological construction of the canonical completion of a cylindric Boolean algebra. Finally, we conclude by discussing how the duality obtained for cylindric Boolean algebras is related to the duality obtained for cylindric ortholattices.

\section{Cylindric ortholattices and orthoframes} 
In this section, we describe some basic properties of cylindric ortholattices and cylindric orthoframes. For more details, consult \cite{harding}. 
 \begin{definition}\label{ol}
An \emph{ortholattice} is an bounded lattice $\langle A;\wedge,\vee,0,1\rangle$ that comes equipped with an additional operation $^{\perp}\colon A\to A$, known as an \emph{orthogonal complementation}, satisfying the following conditions: 
\par
\vspace{-.1cm}
\begin{multicols}{2}
    \begin{enumerate}
        \item $a\wedge a^{\perp}=0$;
        \item $a\vee a^{\perp}=1$;
        \item $a\leq b\Rightarrow b^{\perp}\leq a^{\perp}$; 
        \item $a^{\perp\perp}=a$. 
    \end{enumerate}
\end{multicols}
   \end{definition}

   It is routine to verify that every ortholattice satisfies De Morgan's laws. Moreover, note that unlike Boolean algebras, ortholattices are not in general distributive. In fact, an ortholattice $A$ is a Boolean algebra iff $A$ is distributive.  
    \begin{definition}\label{mol}
A \emph{monadic ortholattice} is an ortholattice $\langle A;\wedge,\vee,^{\perp},0,1\rangle$ that comes equipped with an additional operation $\exists\colon A\to A$, known as a \emph{quantifier}, satisfying the following conditions: 
\par
\vspace{-.1cm}
\begin{multicols}{3}
        \begin{enumerate}
            \item $\exists(a\vee b)=\exists a\vee\exists b$;
               \item $\exists 0=0$; 
               \item $\exists\exists a=\exists a$;
               \item $a\leq \exists a$;
               \item $\exists(\exists a)^{\perp}=(\exists a)^{\perp}$. 
    \end{enumerate}
\end{multicols}
\end{definition}

Any quantifier on an ortholattice $A$ may be viewed as a closure operator whose closed elements $\{a\in A:a=\exists a\}$ form a sub-ortholattice of $A$.

\begin{example}\label{12elements}
 Figure 1 depicts the 12-element monadic ortholattice $\text{mOL}_{12}$. Note that for every $a\in\text{mOL}_{12}$, $\exists a$ is the smallest $\bullet$-element above $a$ so that each $\bullet$-element in $\text{mOL}_{12}$ is a closed element under $\exists$. Note that $\text{mOL}_{12}$ in particular forms a quantum monadic algebra (see \cite[Example 5.7]{harding}).   
\end{example}

 \par
\begin{figure}[htbp]
\begin{center}
    \begin{tikzpicture}
  \node (a) at (0,1.5) {$\bullet$};
  \node (d) at (0,-1.5)  {$\bullet$};
  \node (e) at (-1,-.45)  {$\bullet$};
  \node(f) at (-1,.45) {$\bullet$};
    \node (g) at (1,-.45) {$*$};
  \node (h) at (1,.45) {$*$};
  \node (i) at (2.5,.45) {$*$};
  \node (j) at (2.5,-.45) {$*$};
    \node (k) at (-2.5,.45) {$\bullet$};
  \node (l) at (-2.5,-.45) {$\bullet$};
    \node (m) at (0,.45) {$\bullet$};
  \node (n) at (0,-.45) {$\bullet$};
  \draw (e) -- (d) (g) -- (d) (f) -- (a) (h) -- (a) (a) (i) -- (a) (k) -- (a) (m) -- (a) (j) -- (d) (l) -- (d) (n) -- (d) (j) -- (h) (n) -- (h) (n) -- (f) (m) -- (e) (m) -- (g) (g) -- (i) (e) -- (k) (n) -- (k) (i) -- (n) (j) -- (m) (l) -- (f) (l) -- (m);
  \draw[preaction={draw=white, -,line width=6pt}];
  \end{tikzpicture}\caption{12-element monadic ortholattice}
  \end{center}
\end{figure}
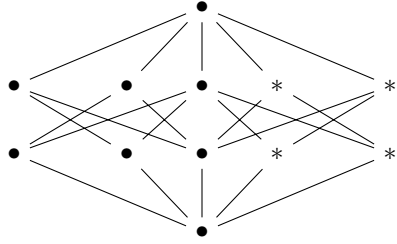

\begin{definition}\label{coll}
 An $I$\emph{-dimensional cylindric ortholattice} is an ortholattice $A$ equipped with a family of unary operators $(\exists_i)_{i\in I}$ and a family of constants $(\delta_{i,k})_{i,k\in I}$ such that for all $i,k,l\in I$, the following conditions are satisfied:
    \begin{enumerate}
        \item $\exists_{i}\colon A\to A$ is a quantifier;
        \item $\exists_{i}\exists_{k}a=\exists_{k}\exists_{i}a$;
        \item $\delta_{i,k}=\delta_{k,i}$ and $\delta_{i,i}=1$;
        \item $i,l\not=k\Rightarrow\exists_{k}(\delta_{i,k}\wedge \delta_{k,l})=\delta_{i,l}$.
    \end{enumerate}
\end{definition}

The reduct $\langle A;\wedge,\vee,^{\perp},0,1,(\exists_{i})_{i\in I}\rangle$ is known as the \emph{I-dimensional diagonal-free cylindric ortholattice reduct} of $A$. 

It is worth pointing out that cylindric Boolean algebras form a strengthening of cylindric ortholattices in two ways. First, the underlying lattice reduct is distributive. Second, there is an axiom asserting that for all $i,k\in I$ such that $i\not=k$, we have $\exists_i(d_{i,k}\wedge a)\wedge\exists_{k}(d_{i,k}\wedge a^{\perp})=0$. This induces an operation of substitution defined by $s^i_k(a)=\exists_i(d_{i,k}\wedge a)$ that is an endomorphism of the Boolean algebra reduct. For details on why this axiom is not included in the definition of cylindric ortholattices, consult \cite[Remark 5.15]{harding}.

Let $X$ be a set and $R$ be a binary relation on $X$. Then for any $U\subseteq X$, we will write $R[U]$ to denote the relational image of $U$ under $R$, i.e., \[R[U]=\{y\in X:xRy\hspace{.1cm}\text{for some}\hspace{.1cm}x\in U\}\] so that for any $x\in X$, we have $R[\{x\}]=\{y\in X:xRy\}$. 
\begin{definition}\label{of}
     An \textit{orthoframe} is pair $\langle X,\perp\rangle$ such that $X$ is a set and $\perp\subseteq X^2$ is an orthogonality relation, i.e., $\perp$ is irreflexive and symmetric. Moreover: 
\begin{enumerate}
    \item for $U\subseteq X$, let $U^{\perp}=\{x\in X: x\perp U\}=\{x\in X:x\perp y\hspace{.2cm}\text{for all}\hspace{.2cm}y\in U\}$  
    \item $U\subseteq X$ is \emph{biorthogonally closed} iff $U=U^{\perp\perp}$. 
\end{enumerate}
For any orthoframe $X$, by $\mathcal{B}(X)$ we denote the collection of biorthogonally closed subsets of $X$. 
\end{definition}
  The following relational structures were introduced in \cite{harding}.  
\begin{definition}\label{mof}
    A \emph{monadic orthoframe} is a triple $\langle X;\perp,R\rangle$ such that: 
    \begin{enumerate}
        \item $\langle X;\perp\rangle$ is an orthoframe;
        \item $R$ is a binary relation on $X$ that is reflexive and transitive; 
        \item $R[R[\{x\}]^{\perp}]\subseteq R[\{x\}]^{\perp}$ for all $x\in X$.  
    \end{enumerate}
\end{definition}
Condition 3 in the above definition can be viewed as the requirement that $R[\{x\}]^{\perp}$ be closed under $R$. 
Note that in any monadic orthoframe $\langle X;\perp,R\rangle$:  \[R[R[\{x\}]^{\perp}]=R[\{x\}]^{\perp}\] for all $x\in X$ by the reflexivity of $R$. 

\begin{definition}\label{col}
An \emph{I-dimensional cylindric orthoframe} is a relational structure $\langle X;\perp,(R_{i})_{i\in I},(\Delta_{i,k})_{i,k\in I}\rangle$ such that for all $i,k,l\in I$:   
    \begin{enumerate}
        \item $\langle X;\perp,R_{i}\rangle$ is a monadic orthoframe; 
        \item $R_i$ commutes with $R_k$;
        \item $\Delta_{i,k}\subseteq X$ with $\Delta_{i,k}=\Delta_{k,i}=\Delta^{\perp\perp}_{i,k}$ and $\Delta_{i,i}=X$;
        \item if $i,l\not=k$, then $R_k[\Delta_{i,k}\cap \Delta_{k,l}]=\Delta_{i,l}$. 
    \end{enumerate}
\end{definition}
The reduct $\langle X;\perp,(R_{i})_{i\in I}\rangle$ is known as the \emph{I-dimensional diagonal-free cylindric orthoframe reduct} of $X$ whenever conditions 1 and 2 are satisfied.

Throughout the remainder of this work, if $X$ is a monadic orthoframe, define: \[U\sqcup V:=(U\cup V)^{\perp\perp},\hspace{.2cm} \exists_RU:=R[U]^{\perp\perp}\] where $R[U]^{\perp\perp}=(R[U])^{\perp\perp}$ for all $U\in\mathcal{B}(X)$. 

\begin{lemma}[\protect{\cite[Proposition 7.31]{harding}}]\label{set representation}
    If $X$ is an $I$-dimensional cylindric orthoframe, then the algebra $\langle\mathcal{B}(X);\cap,\sqcup,^{\perp},\emptyset,X,(\exists_{R_{i}})_{i\in I},(\Delta_{i,k})_{i,k\in I}\rangle$ is a complete $I$-dimensional cylindric ortholattice.  
\end{lemma}
\begin{proof}
    The proof relies on the well-known result of Birkhoff \cite[p. 123]{birkhoff2} that the algebra $\langle\mathcal{B}(X);\cap,\sqcup,^{\perp},\emptyset,X\rangle$ forms a complete ortholattice whenever $X$ is an orthoframe and the result of Harding \cite[Lemma 7.25]{harding} that $\exists_R\colon\mathcal{B}(X)\to\mathcal{B}(X)$ is a quantifier on $\mathcal{B}(X)$ so that the algebra $\langle\mathcal{B}(X);\cap,\sqcup,^{\perp},\emptyset,X,\exists_R\rangle$ is a complete monadic ortholattice whenever $X$ is a monadic orthoframe. 
\end{proof}

\section{Canonical completion}
In this section, using the canonical completion results obtained by Harding, McDonald, and Peinado \cite{harding3} within the setting of monadic ortholattices, we demonstrate that the variety of $I$-dimensional cylindric ortholattices is closed under canonical completions. 

Recall that for a bounded lattice $A$ and a complete lattice $C$, an embedding $e\colon A\to C$ is \emph{dense} if every element in $C$ is both a meet of joins and a join of meets of elements in the image of $A$ under $e$. Moreover, $e$ is \emph{compact} if for all subsets $S,T\subseteq A$, there exists finite subsets $S'\subseteq S$ and $T'\subseteq T$ such that: 
\[\bigwedge e[S]\leq\bigvee e[T]\Longrightarrow\bigwedge e[S']\leq\bigvee e[T'].\]

The \emph{canonical completion} of $A$ is then defined to be a pair $\langle e;C\rangle$ such that $C$ is a complete lattice and $e\colon A\to C$ is an embedding which is dense and compact. It is easy to see that every bounded lattice $A$ has up to isomorphism a unique canonical completion.
\begin{remark}
    Notice that when $B$ is a Boolean algebra, its canonical completion is determined by the embedding of $B$ into the complete powerset Boolean algebra $\wp(X_B)$ where $X_B$ is the Stone space of ultrafilters of $B$.    
\end{remark}

Recall that for a bounded lattice $A$: 
\begin{enumerate}
    \item a non-empty subset $x\subseteq A$ is a \emph{filter} whenever:
    \begin{enumerate}
        \item $a\in x$ and $a\leq b$ implies $b\in x$;
        \item $a\in x$ and $b\in x$ implies $a\wedge b\in x$; 
    \end{enumerate}
    \item a filter $x\subseteq A$ is \emph{proper} whenever $x\not=A$. Note that this is equivalent to the requirement that $0\not\in x$ since $0\leq a$ for all $a\in A$; 
    \item  a proper filter $x\subseteq A$ is \emph{prime} whenever $a\vee b\in x$ implies $a\in x$ or $b\in x$;
    \item a proper filter $x\subseteq A$ is \emph{completely prime} whenever $\bigvee_{i\in I}a_i\in x$ implies $a_i\in x$ for some $i\in I$.
\end{enumerate}

\begin{definition}\label{spectrum}
    Let $A$ be an $I$-dimensional cylindric ortholattice. We define the relational structure $X_A=\langle\mathfrak{F}(A);\perp,(R_i)_{i\in I},(\Delta_{i,k})_{i,k\in I}\rangle$ as follows: 
    \begin{enumerate}
        \item $\mathfrak{F}(A)$ is the collection of all non-empty proper filters of $A$;
        \item $\perp\subseteq\mathfrak{F}(A)^2$ is defined by $x\perp y$ iff there exists $a\in x$ such that $a^{\perp}\in y$; 
        \item $R_{i}\subseteq\mathfrak{F}(A)^2$ is defined by $xR_{i}y$ iff $\exists_{i}[x]\subseteq y$ with $\exists_{i}[x]=\{\exists_{i} a:a\in x\}$; 
        \item $\Delta_{i,k}\subseteq\mathfrak{F}(A)$ is defined by $\Delta_{i,k}=\phi(d_{i,k})=\{x\in\mathfrak{F}(A):\delta_{i,k}\in x\}$.
        \end{enumerate}
        \end{definition}
We will often refer to $X_A$ as the \emph{Goldblatt frame} of $A$. 

        \begin{lemma}\label{lemma 3.3}
            $X_A=\langle\mathfrak{F}(A);\perp,(R_i)_{i\in I},(\Delta_{i,k})_{i,k\in I}\rangle$ is an $I$-dimensional cylindric orthoframe whenever $A$ is an $I$-dimensional cylindric ortholattice.  
        \end{lemma}
        \begin{proof}
The result that $\perp$ is irreflexive and symmetric follows from Goldblatt \cite[Lemma 2]{goldblatt1}. The result that $R_{i}$ is reflexive, transitive, and satisfies the inclusion $R_i[R_i[\{x\}]^{\perp}]\subseteq R_i[\{x\}]^{\perp}$ for each proper filter $x\in\mathfrak{F}(A)$ follows by Harding, McDonald, and Peinado \cite[Proposition 3.5]{harding3}. 

Hence we first verify that the relations $R_{i}$ and $R_{k}$ commute, i.e., \[R_{i}\circ R_{k}=R_{k}\circ R_{i}\hspace{.4cm}\text{for all $i,k\in I$}\] where the relational composition $R_i\circ R_k$ is defined in the usual way by:
\[R_i\circ R_k=\{\langle x,z\rangle\in\mathfrak{F}(A)\times\mathfrak{F}(A):xR_iy\hspace{.1cm}\text{and}\hspace{.1cm}yR_kz\hspace{.1cm}\text{for some}\hspace{.1cm}y\in\mathfrak{F}(A)\}.\]

Assume $xR_{i}y$ and $yR_{k}z$ so that $\exists_{i}[x]\subseteq y$ and $\exists_{k}[y]\subseteq z$ for some $y\in\mathfrak{F}(A)$. This implies that $\exists_k[\exists_i[x]]\subseteq z$ where:
\begin{align*}
    \exists_k[\exists_i[x]]&=\{\exists_kb:b\in\exists_i[x]\}=\{\exists_kb:b=\exists_ic\hspace{.1cm}\text{for some}\hspace{.1cm}c\in x\}
    \\&=\{\exists_k\exists_ic:c\in x\}=\{\exists_i\exists_kc:c\in x\}
\end{align*}
It suffices to find some $w\in\mathfrak{F}(A)$ such that $xR_{k}w$ and $wR_{i}z$. Hence set $w$ to be the proper filter generated by $\exists_k[x]$, i.e.,
\[w=\bigcap\{v\in\mathfrak{F}(A):\exists_k[x]\subseteq v\}.\] It is easy to verify that $w$ is indeed proper since $x$ is proper, and also $\exists_k[x]\subseteq w$ and $\exists_i[w]\subseteq z$ so that $xR_kw$ and $wR_iz$. The right-to-left inclusion follows by a symmetric argument.

Verifying that $\Delta_{i,k}=\Delta_{k,i}$ is straightforward as can be seen below: 
\begin{align*}    \Delta_{i,k}&=\phi(\delta_{i,k})=\{x\in\mathfrak{F}(A):\delta_{i,k}\in x\}\tag{def. of $\Delta_{i,k}$ and $\phi$}
\\&=\{x\in\mathfrak{F}(A):\delta_{k,i}\in x\}\tag{Definition \ref{coll}(3)}
\\&=\phi(\delta_{k,i})=\Delta_{k,i}.\tag{def. of $\Delta_{k,i}$ and $\phi$}
\end{align*}

The result that $\Delta_{i,k}=\Delta_{i,k}^{\perp\perp}$ is immediate by the definition of $\Delta_{i,k}$ and the fact that $\phi(a)$ is bi-orthogonally closed for each $a\in A$: \[\Delta_{i,k}=\phi(\delta_{i,k})=\phi(\delta_{i,k})^{\perp\perp}=\Delta_{i,k}^{\perp\perp}.\] To see $\Delta_{i,i}={\mathfrak{F}(A)}$ is satisfied, note that: 
\begin{align*}    \Delta_{i,i}&=\phi(\delta_{i,i})=\phi(1)
=\{x\in\mathfrak{F}(A)\colon 1\in x\}=\mathfrak{F}(A).
\end{align*}
To verify $R_{k}[\Delta_{i,k}\cap \Delta_{k,l}]=\Delta_{i,l}$ whenever $i,l\not=k$, first note that: 
\begin{align*}
    R_{k}[\Delta_{i,k}\cap \Delta_{k,l}]&=R_{k}[\phi(\delta_{i,k})\cap\phi(\delta_{k,l})]\tag{by the definition of $\Delta_{i,k}$}
    \\&=R_{k}[\phi(\delta_{i,k}\wedge \delta_{k,l})].\tag{by \cite[Theorem 3.14]{mcdonald}}
\end{align*}
 Then for any $y\in R_{k}[\phi(\delta_{i,k}\wedge \delta_{k,l})]$, we have $xR_{k}y$ and hence $\exists_{k}[x]\subseteq y$ for some $x\in\phi(\delta_{i,k}\wedge \delta_{k,l})$. Then $\delta_{i,k}\wedge \delta_{k,l}\in x$ and so $\exists_{k}(\delta_{i,k}\wedge \delta_{k,l})\in y$ but by Definition \ref{coll}(4) we have $\exists_{k}(\delta_{i,k}\wedge \delta_{k,l})=\delta_{i,l}$ so $\delta_{i,l}\in y$ and hence $y\in\phi(\delta_{i,l})=\Delta_{i,l}$ and thus $R_{k}[\phi(\delta_{i,k}\wedge \delta_{k,l})]=R_{k}[\Delta_{i,k}\cap \Delta_{k,l}]\subseteq \Delta_{i,l}$.

Conversely, choose any $y\in \Delta_{i,l}=\phi(\delta_{i,l})$ so that $\delta_{i,l}\in y$. Again, by Definition \ref{coll}(4) we have $\delta_{i,l}=\exists_{k}(\delta_{i,k}\wedge \delta_{k,l})$ so $\exists_{k}(\delta_{i,k}\wedge \delta_{k,l})\in y$. It suffices to find some $x\in\mathfrak{F}(A)$ such that $\exists_{k}[x]\subseteq y$ where $x\in\phi(\delta_{i,k}\wedge \delta_{k,l})$. Hence, set: 
\[x={\uparrow}(\delta_{i,k}\wedge \delta_{k,l})=\{a\in A:\delta_{i,k}\wedge \delta_{k,l}\leq a\}\] Clearly ${\uparrow}(\delta_{i,k}\wedge \delta_{k,l})$ is a proper filter in $A$, in particular, it is the principal filter generated by $\delta_{i,k}\wedge \delta_{k,l}$. Moreover, note by the reflexivity of $\leq$ that it is trivial that $\delta_{i,k}\wedge \delta_{k,l}\in{\uparrow}(\delta_{i,k}\wedge \delta_{k,l})$ so in particular, ${\uparrow}(\delta_{i,k}\wedge \delta_{k,l})\in\phi(d_{i,k}\wedge \delta_{k,l})$. Hence, choose any $a\in \exists_k[{\uparrow}(\delta_{i,k}\wedge\delta_{k,l})]$ so that $a=\exists_kb$ for some $b\in{\uparrow}(\delta_{i,k}\wedge\delta_{k,l})$. If $b=\delta_{i,k}\wedge \delta_{k,l}$, then $\exists_{k}b=\exists_{k}(\delta_{i,k}\wedge \delta_{k,l})$ so $\exists_{k}b=a\in y$. If $b>\delta_{i,k}\wedge \delta_{k,l}$, then $\exists_kb\geq\exists_k(\delta_{i,k}\wedge\delta_{k,l})$ by the monotonicity of $\exists_k$. Then $\exists_kb=a\in y$ since $\exists_k(\delta_{i,k}\wedge\delta_{k,l})\in y$ and $y$ is an upset. Therefore $\exists_{k}[{\uparrow}(\delta_{i,k}\wedge \delta_{k,l})]\subseteq y$ and hence we conclude that $\Delta_{i,l}\subseteq R_{k}[\phi(\delta_{i,k}\wedge \delta_{k,l})]=R_{k}[\Delta_{i,k}\cap \Delta_{k,l}]$, as desired. 
\end{proof}
Conversely to that of Lemma \ref{lemma 3.3}, we also have the following. 
\begin{lemma}\label{lemma 3.4}
    If $A$ is an $I$-dimensional cylindric ortholattice and $X_A$ is its Goldblatt frame, then the algebra $\langle\mathcal{B}(X_A);\cap,\sqcup,^{\perp},\emptyset,\mathfrak{F}(A),(\exists_{R_{i}})_{i\in I},(\Delta_{i,k})_{i,k\in I}\rangle$ is a complete $I$-dimensional cylindric ortholattice. 
\end{lemma}
\begin{proof}
    The result is immediate by Lemma \ref{set representation} and Lemma \ref{lemma 3.3}. 
\end{proof}
 
\begin{lemma}\label{embedding}
     Let $A$ be an $I$-dimensional cylindric ortholattice and let $X_A$ be its Goldblatt frame. Then $\phi\colon A\to\mathcal{B}(X_A)$ defined by $\phi(a)=\{x\in\mathfrak{F}(A):a\in x\}$ exhibits a homomorphic embedding. 
\end{lemma}
\begin{proof}
     By \cite[Theorem 3.14]{mcdonald}, we know that $\phi$ exhibits the desired homomorphic embedding from the ortholattice reduct of $A$ to the ortholattice reduct of $\mathcal{B}(X_A)$. The proof that $\phi$ is a homomorphism for $\exists_i$ for each $i\in I$ runs the same as \cite[Proposition 3.5]{harding3}. Therefore, the result is achieved by noting that $\phi(\delta_{i,k})=\Delta_{i,k}$ by the definition of $\Delta_{i,k}$.
\end{proof}

We now arrive at the main results of this section. 
\begin{theorem}\label{canonical completion}
     Let $A$ be an $I$-dimensional cylindric ortholattice and let $X_A$ be its Goldblatt frame. Then $\langle\phi,\mathcal{B}(X_A)\rangle$ is a canonical completion of $A$. 
\end{theorem}
\begin{proof}
    By Lemma \ref{lemma 3.4}, $\mathcal{B}(X_A)$ forms a complete lattice, and hence it follows by Lemma \ref{embedding} that $\langle\phi;\mathcal{B}(X_A)\rangle$ is a completion of $A$. By \cite[Proposition 3.3]{harding3}, it follows that $\phi$ is a compact and dense embedding from $A$ into $\mathcal{B}(X_A)$ so that $\langle\phi;\mathcal{B}(X_A)\rangle$ is a canonical completion of $A$. 
\end{proof}
\begin{corollary}
    The variety of $I$-dimensional cylindric ortholattices is closed under canonical completions. 
\end{corollary}
\begin{proof}
    By Theorem \ref{canonical completion}, $\langle\phi;\mathcal{B}(X_A)\rangle$ is the canonical completion of $A$. By Lemma \ref{lemma 3.4}, $\mathcal{B}(X_A)$ is an $I$-dimensional cylindric ortholattice.  
\end{proof}
\begin{remark}
If $A$ is a bounded lattice, an element in its canonical completion $A^{\sigma}$ that is a meet of elements in the image of $A$ under the associated embedding $e$ is said to be \emph{closed} and we denote the collection of all closed elements of $A^{\sigma}$ by $\mathcal{K}$. For any bounded lattice that comes equipped with additional monotone operators, there are extensions of the operation for the canonical completion (see \cite{harding2, Gehrke}). These extensions give rise to the \emph{canonical extensions} of the respective algebraic structures. Harding, McDonald, and Peinado \cite{harding3} described these extensions for orthocomplementation and the quantifier as follows:
\[a^{\perp^{\sigma}}=\bigwedge\Bigl\{\bigvee\{e(b^{\perp}):c\leq e(b)\}:c\leq a\hspace{.2cm}\text{and}\hspace{.2cm}c\in \mathcal{K}\Bigr\}\]
\[\exists^{\sigma}a=\bigvee\Bigl\{\bigwedge\{e(\exists b):c\leq e(b)\}:c\leq a\hspace{.2cm}\text{and}\hspace{.2cm}c\in \mathcal{K}\Bigr\}\] where $^{\perp^{\sigma}}$ is the \emph{upper canonical extension} of $^{\perp}$ and $\exists^{\sigma}$ is the \emph{lower canonical extension} of $\exists$. They went on to demonstrate that the variety of ortholattices and monadic ortholattices are closed under canonical extensions.  
\end{remark}

  \section{Monadic and cylindric UVO-spaces}
  We now proceed by characterizing the choice-free topological duals of the $I$-dimensional cylindric ortholattices. Throughout the remainder of this work, given any topological orthoframe $\langle X;\perp,\tau\rangle$ where $\perp\subseteq X^2$ is an orthogonality relation and $\tau\subseteq\wp(X)$ is a topology, let $\mathcal{CO}(X)$ denote the collection of compact open subsets of $X$ and let $\mathcal{COB}(X)=\mathcal{CO}(X)\cap\mathcal{B}(X)$, i.e., let $\mathcal{COB}(X)$ denote the collection of compact open biorthogonally closed subsets of $X$.

Recall that a topological space $X$ is a \emph{$T_0$-space} if for all $x,y\in X$ such that $x\not=y$, there exists an open set $U$ in $X$ such that $x\in U$ and $y\not\in U$. The following combines the upper Vietoris orthospaces studied in \cite{mcdonald} with the monadic orthospaces studied in \cite{harding3}. For a relation $R$, let $\overline{R}$ be its complement. Throughout the remainder of this section, let: \[\exists_R[\mathcal{COB}_X(x)]:=\{\exists_RU:U\in\mathcal{COB}_X(x)\}\]where:\[\mathcal{COB}_X(x)=\{U\in\mathcal{COB}(X):x\in U\}.\]  

\begin{definition}\label{muvo}
A \emph{monadic upper Vietoris orthospace} (henceforth, \emph{monadic UVO-space}) is a relational topological space $\langle X;\perp,R,\tau\rangle$ such that:    
\begin{enumerate}
\item $\langle X;\perp,R\rangle$ is a monadic orthoframe over a $T_0$-space $X$;
\item $\mathcal{COB}(X)$ forms a basis for $X$ and is closed under $\cap$ and $^{\perp}$;
    \item\label{uvo proper filter} every proper filter in $\mathcal{COB}(X)$ is of the form $\mathcal{COB}_X(x)$ for some $x\in X$;
    \item\label{uvo perp} if $x\perp y$, there exists $U\in\mathcal{COB}(X)$ such that $x\in U$ and $y\in U^{\perp}$; 
    \item $R[U]\in\mathcal{COB}(X)$ whenever $U\in\mathcal{COB}(X)$; 
    \item if $x\overline{R}y$, there exists some $U\in\mathcal{COB}(X)$ such that $U=\exists_RV$ for some $V\in\mathcal{COB}_X(x)$ with $y\not\in U$.    
\end{enumerate}
\end{definition}
We regard the topological orthoframe reduct $\langle X;\perp,\tau\rangle$ as the \emph{upper Vietoris orthospace reduct} of $X$ if $X$ is a $T_0$-space satisfying conditions 2 through 4.

\begin{definition}
    Let $X$ be a topological space. Then: 
    \begin{enumerate}
        \item $X$ is a \emph{compact space} if every basic open cover of $X$ has a finite subcover; 
        \item $X$ is a \emph{coherent space} if $\mathcal{CO}(X)$ forms a basis for $X$ and is closed under finite intersections; 
        \item $X$ is a \emph{sober space} provided every completely prime filter in the lattice $\mathcal{CO}(X)$ is of the form $\mathcal{CO}_X(x)$ where $\mathcal{CO}_X(x)=\{U\in\mathcal{CO}(X)\colon x\in U\}$ for some point $x\in X$. 
        \item $X$ is a \emph{spectral space} if $X$ is $T_0$, compact, coherent, and sober. 
    \end{enumerate}
\end{definition}

 Note that unlike Stone spaces, spectral spaces are not in general Hausdorff.  Consult \cite{dickmann} for more details on Stone and spectral spaces.    
\begin{proposition}\label{muvo is spectral}
Every monadic UVO-space is a spectral space.     
\end{proposition}
\begin{proof}
The result follows immediately from \cite[Corollary 5.14]{mcdonald}    
\end{proof}

Recall that in any topological space $X$, the \emph{specialization order} $\leqslant$ of $X$ is defined by $x\leqslant y$ iff $x\in U$ implies $y\in U$ for all open sets $U$ in $X$. The following lemma describes some useful properties of monadic UVO-spaces.   

\begin{lemma}\label{muvo lemma}
    Let $X$ be a monadic UVO-space. Then:
    \begin{enumerate}
    \item the specialization order $\leqslant$ is a partial order;
    \item $x\not\leqslant y$ implies there exists $U\in\mathcal{COB}(X)$ such that $x\in U$ and $y\not\in U$;
        \item $xRy$ implies $\exists_R[\mathcal{COB}_X(x)]\subseteq\mathcal{COB}_X(y)$.
    \end{enumerate}
    
\end{lemma}
\begin{proof}
For part 1, note that the specialization order of any topological space determines a quasi-order. Since by definition every monadic UVO-space is $T_0$, it follows that $\leqslant$ is anti-symmetric and hence a partial order. 

Part 2 follows immediately from the fact that $\mathcal{COB}(X)$ forms a basis: Assume that $x\not\leqslant y$. By the definition of $\leqslant$, there exists an open set $U$ in $X$ such that $x\in U$ but $y\not\in U$. By Definition \ref{muvo}(2) we obtain: \[U=\bigcup^n_{i=1}V_i\hspace{.2cm}\text{for}\hspace{.2cm}V_1,\dots, V_n\in\mathcal{COB}(X).\]     
Thus $x\in V_i$ for some $i$ such that $1\leq i\leq n$ where $V_i\in\mathcal{COB}(X)$ and $y\not\in V_i$ for each $V_i$. For part 3, assume $xRy$ and let $U\in\exists_R[\mathcal{COB}_X(x)]$ so that $U=\exists_RV$ for some $V\in\mathcal{COB}_X(x)$. This implies $x\in V$ by the definition of $\mathcal{COB}_X$ and also $R[V]\in\mathcal{COB}(X)$ by Definition \ref{muvo}(5). Thus by the definition of $\exists_R$ and the fact that $R[V]\in\mathcal{B}(X)$, we obtain: \[\exists_RV=R[V]^{\perp\perp}=R[V]=\{z\in X:wRz\hspace{.1cm}\text{for some}\hspace{.1cm}w\in V\}\] and hence we find that $y\in\exists_RV$ since $x\in V$ and $xRy$ by hypothesis. This implies that $\exists_RV=U\in\mathcal{COB}_X(y)$, as desired.   
\end{proof}

By extending monadic UVO-spaces with the relational structure of an $I$-dimensional cylindric orthoframe, we arrive at the following. 
\begin{definition}\label{cuvo}
    An \emph{I-dimensional cylindric UVO-space} is a topological space $\langle X;\perp,(R_i)_{i\in I},(\Delta_{i,k})_{i,k\in I},\tau\rangle$ satisfying the following conditions for all $i,k\in I$:  
    \begin{enumerate}
        \item $\langle X;\perp,(R_i)_{i\in I},(\Delta_{i,k})_{i,k\in I}\rangle$ is an $I$-dimensional cylindric orthoframe;
         \item $\langle X;\perp,(R_i)_{i\in I},\tau\rangle$ is a monadic UVO-space; 
        \item $\Delta_{i,k}\in\mathcal{CO}(X)$.  
    \end{enumerate}
\end{definition}
We call $\langle X;\perp,(R_{i})_{i\in I},\tau\rangle$ the \emph{I-dimensional diagonal-free cylindric UVO-space reduct} of $X$ whenever condition 2 is satisfied. 

Notice that we only require that $\Delta_{i,k}\in\mathcal{CO}(X)$ since $\Delta_{i,k}=\Delta_{i,k}^{\perp\perp}$ by condition 3 of Definition \ref{col}, which implies $\Delta_{i,k}\in\mathcal{COB}(X)$.

 \section{Topological representation}
The following endows the canonical completion of a cylindric ortholattice with the spectral topology along the lines of \cite{mcdonald,bezhanishvili}. 
\begin{definition}\label{spectrum}
    Let $A$ be an $I$-dimensional cylindric ortholattice. We define the topological space $\mathcal{S}_0(A)=\langle\mathfrak{F}(A);\perp,(R_i)_{i\in I},(\Delta_{i,k})_{i,k\in I},\tau(\beta)\rangle$ as follows: 
    \begin{enumerate}
        \item $\langle \mathfrak{F}(A);(R_i)_{i\in I},(\Delta_{i,k})_{i,k\in I}\rangle$ is the Goldblatt frame of $A$;
        \item $\tau(\beta)\subseteq\wp(\mathfrak{F}(A))$ is the topology generated by the basis $\beta=\{\phi(a):a\in A\}$ where $\phi\colon A\to\wp(\mathfrak{F}(A))$ is defined by $\phi(a)=\{x\in\mathfrak{F}(A):a\in x\}$.
    \end{enumerate}
\end{definition}

It is clear in the above definition that every basic open set in $\mathcal{S}_0(A)$ is of the form $\phi(a)$ for some $a\in A$.

The following results are well-known. 
\begin{theorem}[Prime Ideal Theorem]
    Let $A$ be a Boolean algebra. Then every proper ideal in $A$ is contained within some prime ideal.   
\end{theorem}
\begin{proof}
    The result follows by way of the Axiom of Choice (see \cite{herrlich}).   
\end{proof}
\begin{theorem}[Alexander's Subbase Theorem]
Let $X$ be a topological space and let $\sigma$ be a subbasis of $X$. Then $X$ is a compact topological space if and only if every basic open cover of $X$ by elements of $\sigma$ has a finite subcover. 
\end{theorem}
\begin{proof}
The result follows by the Boolean Prime Ideal Theorem. See for instance Davey and Priestley \cite[p. 278--279]{davey}.  
\end{proof}

\begin{remark}
    Goldblatt \cite[Proposition 3]{goldblatt1} showed that for any ortholattice $A$, equipping $\langle\mathfrak{F}(A);\perp\rangle$ with the topology $\tau(\sigma)$ with sub-basis:
\[\sigma=\{\phi(a):a\in A\}\cup\{\mathfrak{F}(A)\setminus\phi(a):a\in A\}\]
 makes $\langle\mathfrak{F}(A);\tau(\sigma)\rangle$ a Stone space, i.e., a totally-disconnected compact space. Goldblatt's proof of compactness makes use of Alexander's Subbase Lemma Theorem, which is strictly weaker than the Prime Ideal Theorem for Boolean algebras, and holds in ZFC but not ZF alone. Therefore Goldblatt's topological representation of ortholattices is non-constructive. 
\end{remark}

It is important to note that the use of the Axiom of Choice is in fact required under Goldblatt's representation and is hence not obtainable in ZF alone, as is seen through the subsequent proposition.

\begin{proposition}[\protect{\cite[Proposition 3.12]{mcdonald}}]\label{requires choice}
  The following are equivalent:
  \begin{enumerate}
  \item\label{choice1} the Prime Ideal Theorem for Boolean algebras;   
  \item\label{choice2} The space $\langle\mathfrak{F}(A);\tau(\sigma)\rangle$ is compact for all Boolean algebras $A$.
  \end{enumerate}
\end{proposition}
\begin{proof}
    The proof of the $(1)\Longrightarrow(2)$ direction is the result \cite[Proposition 3]{goldblatt1} mentioned above. The proof of the  $(2)\Longrightarrow(1)$ direction relies on Rado's Weak Selection Lemma \cite{howard}. 
\end{proof}

\begin{lemma}\label{algebra to space}
    If $A$ is an $I$-dimensional cylindric ortholattice, then $\mathcal{S}_0(A)=\langle \mathfrak{F}(A);\perp,(R_{i})_{i\in I},(\Delta_{i,k})_{i,k\in I},\tau(\beta)\rangle$ is an $I$-dimensional cylindric UVO-space whose specialization order is given by set-theoretic inclusion. 
\end{lemma}
\begin{proof}
Conditions 1 and 2 of Definition \ref{cuvo} follow from Lemma \ref{lemma 3.3}, Lemma \ref{embedding}, \cite[Lemma 4.3]{mcdonald} and the following, which verifies that condition 6 of monadic UVO-spaces is satisfied: assume $x\overline{R}y$. It suffices to find some $a\in A$ such that $\phi(a)=\phi(\exists_i b)$ for some $b\in A$ with $\phi(b)\in\mathcal{COB}_{\mathcal{S}_0(A)}(x)$ and $y\not\in\phi(a)$ where: \[\mathcal{COB}_{\mathcal{S}_0}(x):=\{\phi(c)\in\mathcal{COB}(\mathcal{S}_0(A)):x\in\phi(c)\}\] By hypothesis, we have $\exists_{R_i}[x]\not\subseteq y$ and hence there exist some $a\in A$ such that $a\in\exists_{R_i}[x]$ but $a\not\in y$. Hence, we have $a=\exists_i b$ for some $b\in x$. Therefore, $x\in\phi(b)$ so $\phi(b)\in\mathcal{COB}_{\mathcal{S}_0}(x)$ and $y\not\in\phi(a)=\phi(\exists_i b)$, as desired. For condition 3 of Definition \ref{cuvo}, it is clear that $\Delta_{i,k}=\phi(\delta_{i,k})\in\mathcal{COB}(\mathcal{S}_0(A))$. 
\end{proof}

\begin{lemma}\label{space to algebra}
   If $X$ is an $I$-dimensional cylindric UVO-space, then the algebra $\mathcal{A}_0(X)=\langle\mathcal{COB}(X);\cap,\sqcup,^{\perp},\emptyset,X,(\exists_{R_{i}})_{i\in I},(\Delta_{i,k})_{i,k\in I}\rangle$ is an $I$-dimensional cylindric ortholattice. 
\end{lemma}
\begin{proof}
   The result follows by Lemma \ref{set representation}, Lemma \ref{algebra to space}, and \cite[Lemma 4.2]{mcdonald}.  
\end{proof}
\begin{theorem}\label{topological representation}
    Every $I$-dimensional cylindric ortholattice $A$ is isomorphic to $\mathcal{A}_0(\mathcal{S}_0(A))$ under $\phi$. 
\end{theorem}
\begin{proof}
    By \cite[Theorem 3.14]{mcdonald}, we know that $\phi$ exhibits the desired isomorphism from the ortholattice reduct of $A$ to the ortholattice reduct of $\mathcal{A}_0(\mathcal{S}_0(A))$. 
    
    The proof that $\phi$ is a homomorphism for $\exists_i$ for each $i\in I$ runs the same as \cite[Proposition 5.3]{harding3}. Therefore, the result is achieved by noting that $\Delta_{i,k}=\phi(\delta_{i,k})$ by the definition of $\Delta_{i,k}$. Lastly, $\phi[A]=\mathcal{COB}(\mathcal{S}_0(A))$. 
\end{proof}

If $\langle X; R_X,Q_X,\tau_X\rangle$ and $\langle Y; R_Y,Q_Y, \tau_Y\rangle$ are topological spaces with binary relations $R_X,Q_X\subseteq X\times X$ and $R_Y,Q_Y\subseteq Y\times Y$, then $f\colon X\to Y$ is a \emph{relational homeomorphism} if $f$ is a homeomorphism and a relational isomorphism i.e., for all $x,y\in X$, we have $xR_Xy\Leftrightarrow f(x)R_Yf(y)$ and $xQ_Xy\Leftrightarrow f(x)Q_Yf(y)$. 
\begin{theorem}\label{algebraic realization}
    Every $I$-dimensional cylindric UVO-space $X$ is relationally homeomorphic to $\mathcal{S}_0(\mathcal{A}_0(X))$. 
\end{theorem}
\begin{proof}
 By \cite[Theorem 4.4]{mcdonald}, it follows that the continuous function: \[\psi\colon X\to\mathcal{S}_0(\mathcal{A}_0(X));\hspace{.2cm}\psi(x)=\{U\in\mathcal{COB}(X):x\in U\}\] exhibits a relational homeomorphism from the UVO-space reduct of $X$ to the UVO-space reduct of $\mathcal{S}_0(\mathcal{A}_0(X))$.
It is an easy exercise to verify that $\psi(x)$ is a proper filter in $\mathcal{A}_0(X)$ so that $\psi$ is indeed a well-defined function. Injectivity follows from Lemma \ref{muvo lemma}(2) and surjectivity follows from Definition \ref{muvo}(3).

The result that $\psi$ exhibits a relational isomorphism from the monadic UVO-space reduct $\langle X;R_i\rangle$ of $X$ to the monadic UVO-space reduct $\langle\mathfrak{F}(\mathcal{COB}(X));R_i\rangle$ of $\mathcal{S}_0(\mathcal{A}_0(X))$ for each $i\in I$ runs as follows. Assume that $xR_iy$ for any $i\in I$. Then by Lemma \ref{muvo lemma}, we have $\exists_{R_i}[\mathcal{COB}_X(x)]\subseteq\mathcal{COB}_X(y)$ so 
$\exists_{R_i}[\{U\in\mathcal{COB}(X):x\in U\}]\subseteq\{U\in\mathcal{COB}(X):y\in U\}$ by the definition of $\mathcal{COB}_X(x)$ and $\mathcal{COB}_X(y)$. Then by the definition of $\psi$, we have $\exists_{R_i}[\psi(x)]\subseteq \psi(y)$ and hence $\psi(x)R_i\psi(y)$ given the fact that $\psi(x)$ and $\psi(y)$ are proper filters in $\mathcal{A}_0(X)$ and the definition of $R_i$. For the converse implication, assume for the sake of contraposition that $x\overline{R}_iy$. Then by condition 6 of Definition \ref{muvo}, there exists some $U\in\mathcal{COB}(X)$ such that $U=\exists_{R_i}V$ for some $V\in\mathcal{COB}_X(x)$ with $y\not\in U$. Then $U\in\exists_{R_i}[\mathcal{COB}_X(x)]$ but $U\not\in\mathcal{COB}_X(y)$ and hence $\exists_{R_i}[\mathcal{COB}_X(x)]\not\subseteq\mathcal{COB}_X(y)$ so $\exists_{R_i}[\psi(x)]\not\subseteq\psi(y)$ and hence $\psi(x)\overline{R}\psi(y)$. This completes the proof. 
  \end{proof}

\section{Duality}
In this section, we introduce suitable spectral frame morphisms to accompany the cylindric UVO-spaces and then prove that the resulting category is dually equivalent to the category of cylindric ortholattices and homomorphisms. 
\begin{definition}
    Let $A$ and $A'$ be $I$-dimensional cylindric ortholattices. Then a function $h\colon A\to A'$ is a \emph{homomorphism} provided: 
    \begin{enumerate}
        \item $h$ is a bounded lattice homomorphism; 
        \item $h(a^{\perp})=h(a)^{\perp}$; 
        \item $h(\exists_{i}a)=\exists_{i}h(a)$; 
        \item $h(\delta_{i,k})=\delta'_{i,k}$. 
    \end{enumerate}
\end{definition}

By $\mathcal{COL}$ we denote the category of $I$-dimensional cylindric ortholattices and their homomorphisms.

Recall that if $X$ and $Y$ are spectral spaces, a \emph{spectral map} is a function $f\colon X\to Y$ such that $f^{-1}[U]\in\mathcal{CO}(X)$ for all $U\in\mathcal{CO}(Y)$.

The following combines the spectral frame morphisms studied in \cite{mcdonald} with the continuous frame morphisms studied in \cite{harding3}. 
\begin{definition}\label{cuvo map}
    Let $X$ and $X'$ be $I$-dimensional cylindric UVO-spaces. Then a map $f\colon X\to X'$ is an \emph{I-dimensional cylindric UVO-map} if for all $i,k\in I$:   
    \begin{enumerate}
    \item $f$ is a spectral map;
      \item if $x\not\perp y$, then $f(x)\not\perp f(y)$;
      \item if $z \not \perp f(y)$,
  there exists $x \in X$ such that $x \not \perp y$ and $z \leqslant f(x)$; 
   \item $R_{i}[f^{-1}[U]]=f^{-1}[R_{i}[U]]$ for all $U\in\mathcal{COB}(X)$; 
   \item $f^{-1}[\Delta'_{i,k}]=\Delta_{i,k}$. 
  \end{enumerate}
\end{definition}

We denote by $\mathcal{CUVO}$ the category of $I$-dimensional cylindric UVO-spaces and $I$-dimensional cylindric UVO-maps. 

\begin{proposition}\label{prop1 spectral map}
 If $X$ and $X'$ are UVO-spaces and $f\colon X\to X'$ is a UVO-map, then $f^{-1}[U]\in\mathcal{COB}(X)$ for each $U\in\mathcal{COB}(X')$.  
 \end{proposition}
 \begin{proof}
     Bimb\'o \cite{bimbo} showed that analogously defined continuous frame morphisms satisfying conditions 2 and 3 of Definition \ref{cuvo map} satisfy: \[f^{-1}[U]=f^{-1}[U^{\perp\perp}]=f^{-1}[U^{\perp}]^{\perp}=f^{-1}[U]^{\perp\perp}\] for any $U\in\mathcal{B}(X')$. This, along with our hypothesis that $f$ is spectral, implies that $f^{-1}[U]\in\mathcal{COB}(X)$ whenever $U\in\mathcal{COB}(X')$. 
 \end{proof}
\begin{lemma}\label{continuous to homomorphism}
    Let $X$ and $X'$ be $I$-dimensional cylindric UVO-spaces and let $f\colon X\to X'$ be an $I$-dimensional cylindric UVO-map. Then the operation $\mathcal{A}_1(f)\colon \mathcal{A}_0(X')\to\mathcal{A}_0(X)$ defined by $\mathcal{A}_1(f):=f^{-1}$ is an $I$-dimensional cylindric ortholattice homomorphism. Therefore $\mathcal{A}_F=\langle\mathcal{A}_0,\mathcal{A}_1\rangle\colon\mathcal{CUVO}\to\mathcal{COL}$ determines a contravariant functor.  
\end{lemma}
\begin{proof}
    Let $X$ and $X'$ be monadic orthospaces. By \cite[Theorem 5.12]{mcdonald}, it follows that $\mathcal{A}_0(f)\colon\mathcal{A}_0(X')\to\mathcal{A}_0(X)$ is an ortholattice homomorphism whenever $f\colon X\to X'$ is a UVO-map. We now verify that $\mathcal{A}_1(f)$ maps $\exists_i$ homomorphically to $\exists_{R_i}$. The proof runs similarly to \cite[Lemma 5.7]{harding3}. 
    \begin{align*}
        \mathcal{A}_1(f)[\exists_{R_{i}}U]&=f^{-1}[\exists_{R_{i}}U]=f^{-1}[R_{i}[U]^{\perp\perp}]\tag{definition of $\mathcal{A}_1$ and $\exists_{R_{i}}$}
        \\&=f^{-1}[R_{i}[U]]=R_{i}[f^{-1}[U]]\tag{$R_{i}[U]\in\mathcal{B}(X)$, Def. \ref{cuvo map}(4)}
        \\&=R_{i}[f^{-1}[U]]^{\perp\perp}=\exists_{R_{i}} f^{-1}[U]\tag{$R_{i}[f^{-1}[U]]\in\mathcal{B}(X)$}\\&=\exists_{R_{i}}\mathcal{A}_1(f)[U]\tag{by the definition of $\mathcal{A}_1$}
    \end{align*}

Finally, we note that the result that $\mathcal{A}_F$ is a homomorphism for $\Delta_{i,k}$ follows directly from the definition of $\mathcal{A}_1$ and Definition \ref{cuvo map}(5) as follows: \[\mathcal{A}_1(f)[\Delta'_{i,k}]=f^{-1}[\Delta'_{i,k}]=\Delta_{i,k}\]

    Hence $\mathcal{A}_0(f)\colon\mathcal{A}_0(X')\to\mathcal{A}_0(X)$ is an $I$-dimensional cylindric ortholattice homomorphism whenever $f\colon X\to X'$ is an $I$-dimensional cylindric UVO-map. Moreover, by Lemma \ref{space to algebra} it follows that $\mathcal{A}_F=\langle\mathcal{A}_0,\mathcal{A}_1\rangle\colon\mathcal{CUVO}\to\mathcal{COL}$ is a contravariant functor. 
\end{proof}

Conversely to Lemma \ref{continuous to homomorphism}, we now show that every homomorphism between cylindric ortholattices induces a cylindric UVO-map between their duals. 
\begin{lemma}\label{functors}
    Let $A$ and $A'$ be $I$-dimensional cylindric ortholattices and let $h\colon A\to A'$ be an $I$-dimensional cylindric ortholattice homomorphism. Then the map $\mathcal{S}_1(h)\colon\mathcal{S}_0(A')\to\mathcal{S}_0(A)$ defined by $\mathcal{S}_1(h):=h^{-1}$ is an $I$-dimensional cylindric UVO-map. Therefore $\mathcal{S}_F=\langle\mathcal{S}_0,\mathcal{S}_1\rangle\colon\mathcal{COL}\to\mathcal{CUVO}$ determines a contravariant functor.  
\end{lemma}
\begin{proof}
    Let $A$ and $A'$ be monadic ortholattices. By \cite[Theorem 5.12]{mcdonald}, it follows that $\mathcal{S}_1(h)\colon\mathcal{S}_0(A')\to\mathcal{S}_0(A)$ is a UVO-map whenever $h\colon A\to A'$ is an ortholattice homomorphism. Since the canonical representation map $\phi\colon A\to\mathcal{A}_0(\mathcal{S}_0(A))$ is a homomorphism for $\exists$, i.e., $\exists$ is homomorphically mapped to $\exists_R$, then by the definition of $\exists_R$ and our hypothesis that $h$ is a monadic ortholattice homomorphism, the result follows analogously to that of \cite[Lemma 5.7]{harding3}: 
    \begin{align*}
        R_{i}[\mathcal{S}_1(h)^{-1}[U]]&=R_{i}[(h^{-1})^{-1}[\phi(a)]]\tag{by the definition of $\mathcal{S}_1$}
        \\&=(h^{-1})^{-1}[R_{i}[\phi(a)]]\tag{Def \ref{cuvo map}(4), $(h^{-1})^{-1}$ a UVO-map}
        \\&=\mathcal{S}_1(h)^{-1}[R_{i}[U]]\tag{by the definition of $\mathcal{S}_1$}
    \end{align*}
  Therefore $\mathcal{S}_1(h)\colon\mathcal{S}_0(A')\to\mathcal{S}_0(A)$ is a monadic UVO-map whenever $h\colon A\to A'$ is a monadic ortholattice homomorphism.
  
  We now verify that $\mathcal{S}_1(h)^{-1}[\Delta_{i,k}]=\Delta'_{i,k}$. The result follows from the fact that $h$ is a homomorphism for diagonal elements and some general principles: 
  \begin{align*}
      \mathcal{S}_1(h)^{-1}[\Delta_{i,k}]&=(h^{-1})^{-1}[\Delta_{i,k}]=(h^{-1})^{-1}[\phi(\delta_{i,k})]\tag{def of $\mathcal{S}_1$ and $\Delta_{i,k}$}
      \\&=\{x\in\mathfrak{F}(A'): h^{-1}[x]\in\phi(\delta_{i,k})\}\tag{definition of $(h^{-1})^{-1}$}
      \\&=\{x\in\mathfrak{F}(A'):\delta_{i,k}\in h^{-1}[x]\}\tag{definition of $\phi$}
      \\&=\{x\in\mathfrak{F}(A'):h(\delta_{i,k})\in x\}\tag{definition of $h^{-1}$}
       \\&=\{x\in\mathfrak{F}(A'):\delta'_{i,k}\in x\}\tag{$h$ being a homomorphism}
       \\&=\phi(\delta'_{i,k})=\Delta'_{i,k}\tag{definition of $\phi$ and $\Delta'_{i,k}$}
  \end{align*}
Therefore $\mathcal{S}_1(h)\colon\mathcal{S}_0(A')\to\mathcal{S}_0(A)$ is an $I$-dimensional cylindric UVO-map whenever $h\colon A\to A'$ is a homomorphism of $I$-dimensional cylindric ortholattices. 
  
  Finally, by Lemma \ref{algebra to space}, it follows that $\mathcal{S}_F=\langle\mathcal{S}_0,\mathcal{S}_1\rangle\colon\mathcal{COL}\to\mathcal{CUVO}$ is a contravariant functor.
\end{proof}


We now arrive at the main result of this paper. 
\begin{theorem}
    The contravariant functors $\mathcal{A}_F=\langle\mathcal{A}_0,\mathcal{A}_1\rangle\colon\mathcal{CUVO}\to\mathcal{COL}$ and $\mathcal{S}_F=\langle\mathcal{S}_0,\mathcal{S}_1\rangle\colon\mathcal{COL}\to\mathcal{CUVO}$ constitute a dual equivalence between the categories $\mathcal{COL}$ and $\mathcal{CUVO}$.  
\end{theorem}
 \begin{proof}
    First note that the topological representation provided by Theorem \ref{topological representation} and the algebraic realization provided by Theorem \ref{algebraic realization} along with Lemmas \ref{continuous to homomorphism} and \ref{functors}. imply that the following diagrams commute:  
    \par
    \vspace{.1cm}
\begin{tikzcd}
A \arrow[dd, "h"'] \arrow[r] \arrow[r] \arrow[rr, bend left, "\phi"] & \mathcal{S}_0(A) \arrow[r]                                                & \mathcal{A}_0(\mathcal{S}_0(A)) \arrow[dd, "\mathcal{A}_1(\mathcal{S}_1(h))"] \\
{} \arrow[r, "1-1", no head, dashed]                                       & {} \arrow[r, "1-1", no head, dashed]                                                 & {}                                                                                          \\
A' \arrow[r] \arrow[rr, bend right, "\phi"']                             & \mathcal{S}_0(A') \arrow[uu, "\mathcal{S}_1(h)" description] \arrow[r] & \mathcal{A}_0(\mathcal{S}_0(A'))                                                
\end{tikzcd}
\hskip .2em
\begin{tikzcd}
X \arrow[dd, "f"'] \arrow[r] \arrow[rr, bend left, "\psi"] & \mathcal{A}_0(X) \arrow[r]                                                & \mathcal{S}_0(\mathcal{A}_0(X)) \arrow[dd, "\mathcal{S}_1(\mathcal{A}_1(f))"] \\
{} \arrow[r, "1-1", no head, dashed]                            & {} \arrow[r, "1-1", no head, dashed]                                                & {}                                                                                         \\
X' \arrow[r] \arrow[rr, bend right, "\psi"']                   & \mathcal{A}_0(X') \arrow[uu, "\mathcal{A}_1(f)" description] \arrow[r] & \mathcal{S}_0(\mathcal{A}_0(X'))                              
\end{tikzcd}
\noindent where $\mathcal{A}_1(\mathcal{S}_1(h))=(h^{-1})^{-1}$ and $\mathcal{S}_1(\mathcal{A}_1(f))=(f^{-1})^{-1}$ since $\mathcal{S}_1(h)=h^{-1}$ and $\mathcal{A}_1(f)=f^{-1}$. The dotted lines in the above commutative diagrams labelled \say{$1-1$} indicate the $1-1$ (dual) correspondence that has been established between homomorphisms of cylindric ortholattices and cylindric UVO-maps. 

Moreover, since $\phi$ is always an isomorphism and $\psi$ is always a relational homeomorphism, this means that for any $I$-dimensional cylindric ortholattice $A$ and $I$-dimensional cylindric UVO-space $X$, the components $\eta_A\colon \text{id}_A(A)\to \mathcal{A}_0(\mathcal{S}_0(A))$ and $\theta_X\colon \text{id}_X(X)\to \mathcal{S}_0(\mathcal{A}_0(X))$ of the respective natural transformations $\eta\colon \text{id}_{\mathcal{COL}}\to\mathcal{A}_F\circ\mathcal{S}_F$ and $\theta\colon\text{id}_{\mathcal{CUVO}}\to\mathcal{S}_F\circ\mathcal{A}_F$ are isomorphisms. Therefore $\eta=\phi$ and $\theta=\psi$ are natural isomorphisms. This completes the proof that the contravariant functors $\mathcal{A}_F\colon\mathcal{CUVO}\to\mathcal{COL}$ and $\mathcal{S}_F\colon\mathcal{COL}\to\mathcal{CUVO}$ determine a dual equivalence of categories between $\mathcal{COL}$ and $\mathcal{CUVO}$.  
 \end{proof}
\section{Cylindric Boolean algebras}
 In this final section, we establish canonical completion and choice-free duality results for cylindric Boolean algebras and conclude with a discussion of how these results are related to the canonical completion and choice-free duality results that have been established for cylindric ortholattices in the previous sections. Although the definition of cylindric Boolean algebras was hinted at within Section 2, we give an explicit definition for the sake of clarity. We note that quantifiers on a Boolean algebra are defined the same way as they are on an ortholattice.
\begin{definition}
 An $I$\emph{-dimensional cylindric Boolean algebra} is a Boolean algebra $A$ equipped with a family of unary operators $(\exists_i)_{i\in I}$ and a family of constants $(\delta_{i,k})_{i,k\in I}$ such that for all $i,k,l\in I$:
    \begin{enumerate}
        \item $\exists_{i}\colon A\to A$ is a quantifier;
        \item $\exists_{i}\exists_{k}a=\exists_{k}\exists_{i}a$;
        \item $\delta_{i,k}=\delta_{k,i}$ and $\delta_{i,i}=1$; 
        \item $k\not=i,l\Rightarrow\delta_{i,l}=\exists_{k}(\delta_{i,k}\wedge \delta_{k,l})$; 
        \item $i\not=k\Rightarrow\exists_i(d_{i,k}\wedge a)\wedge\exists_{k}(d_{i,k}\wedge a^{\perp})=0$. 
    \end{enumerate}
\end{definition}

We now introduce some key aspects of the choice-free duality for Boolean algebras established by Bezhanishvili and Holliday \cite{bezhanishvili}, which is based on the choice-free representation of Boolean algebras developed by Holliday \cite{holliday} and makes use of the concept of a possibility frame. A \emph{possibility frame} is a triple $\langle X;\leq,P\rangle$ such that $\langle X;\leq\rangle$ is a poset and $P$ is the collection of regular open subsets in the upset topology on $\langle X;\leq\rangle$ where $P$ contains $X$ and is closed under intersection and the operation $^*$ defined by: 
\[U^*=\{x\in X:x'\not\in U\hspace{.2cm}\text{for all $x'\geq x$}\}.\]
Recall that if $X$ is a topological space, then an open set $U\subseteq X$ is \emph{regular open} if and only if $\text{Int}(\text{Cl}(U))=U$ where $\text{Int}$ and $\text{Cl}$ are the interior and closure operations on $X$. Moreover, in the upset topology $\text{UP}(X;\leq)$, we have: 
\[\text{Int}_{\leq}(U)=\{x\in X:y\in U\hspace{.2cm}\text{for all $y\geq x$}\}\]
\[\text{Cl}_{\leq}(U)=\{x\in X:y\in U\hspace{.2cm}\text{for some $y\geq x$}\}\] and thus $\text{Int}_{\leq}(U)=X\setminus\text{Cl}_{\leq}(X\setminus U)$.

Therefore, an open set $U$ is \emph{regular open} in the upset topology on $X$ if and only if $\text{Int}_{\leq}(\text{Cl}_{\leq}(U))=U$. Moreover, since $U^*=\text{Int}_{\leq}(X\setminus U)$, it follows that an open set $U$ is \emph{regular open} in $\text{UP}(X;\leq)$ if and only if $U^{**}=U$. We will denote the collection of all such regular opens by $\mathcal{REG}(X)$ so that $\mathcal{COREG}(X)=\mathcal{CO}(X)\cap\mathcal{REG}(X)$, i.e., $\mathcal{COREG}(X)$ is the collection of all compact open subsets of $X$ that are regular open in the upset topology on $X$.

The following class of spaces was introduced by Bezhanishvili and Holliday \cite{bezhanishvili} and arise as the choice-free duals of the Boolean algebras. Let: 
\[\mathcal{COREG}_X(x)=\{U\in\mathcal{COREG}(X):x\in U\}.\]
\begin{definition}\label{uv space}
    A \emph{UV-space} is a $T_0$-space $X$ satisfying the following conditions: 
    \begin{enumerate}
        \item $\mathcal{COREG}(X)$ is closed under $\cap$, $\text{Int}_{\leqslant}(X\setminus\cdot)$ and is a basis for $X$; 
        \item every proper filter in $\mathcal{COREG}(X)$ is $\mathcal{COREG}_X(x)$ for some $x\in X$. 
    \end{enumerate}
\end{definition}

The following are relational extensions of the UV-spaces described above.

\begin{definition}\label{cylindric uv space}
    Let $\exists_{S_i}[\mathcal{COREG}_X(x)]:=\{\exists_{S_i}U:U\in\mathcal{COREG}_X(x)\}$. An \emph{I-dimensional cylindric UV-space} is a topological space of the following shape $\langle X;(S_i)_{i\in I},(\Delta_{i,k})_{i,k\in I},\tau\rangle$ such that for all $i,k,l\in I$: 
    \begin{enumerate}
        \item $\langle X;\tau\rangle$ is a UV-space; 
        \item $S_i$ is an equivalence relation; 
         \item $S_i$ commutes with $S_k$;
        \item $\Delta_{i,k}\subseteq X$ with $\Delta^{**}_{i,k}=\Delta_{i,k}=\Delta_{k,i}$ and $\Delta_{i,i}=X$;
        \item if $i,l\not=k$, then $S_k[\Delta_{i,k}\cap \Delta_{k,l}]=\Delta_{i,l}$;
        \item if $i\not=k$, then $S_i[\Delta_{i,k}\cap U]\cap S_i[\Delta_{i,k}\cap U^*]=\emptyset$; 
        \item $S_i[U]\in\mathcal{COREG}(X)$ whenever $U\in\mathcal{COREG}(X)$; 
        \item if $x\overline{S}_iy$, then $\exists_{S_i}[\mathcal{COREG}_X(x)]\not=\exists_{S_i}[\mathcal{COREG}_X(y)]$.  
        \item $\Delta_{i,k}\in\mathcal{CO}(X)$. 
    \end{enumerate}
\end{definition}
It is obvious that for any $I$-dimensional cylindric UV-space $X$, we have $\Delta_{i,k}\in\mathcal{COREG}(X)$ for all $i,k\in I$ by conditions 4 and 9 of Definition \ref{cylindric uv space}. 
\begin{definition}\label{canonical frame of a cylindric ba}
    Let $A$ be an $I$-dimensional cylindric Boolean algebra. Then define the topological space $\mathcal{F}_0(A)=\langle \mathfrak{F}(A);(S_i)_{i\in I},(\Delta_{i,k})_{i,k\in I},\tau(\beta)\rangle$ as follows: 
    \begin{enumerate}
       \item $\mathfrak{F}(A)$ is the collection of non-empty proper filters of $A$; 
       \item $S_i\subseteq\mathfrak{F}(A)\times\mathfrak{F}(A)$ is defined by $xS_iy\Longleftrightarrow\exists_i[x]=\exists_i[y]$ for all $i\in I$; 
      \item $\Delta_{i,k}\subseteq\mathfrak{F}(A)$ is defined by $\Delta_{i,k}=\phi(\delta_{i,k})=\{x\in\mathfrak{F}(A):\delta_{i,k}\in x\}$; 
      \item  $\Omega(\beta)\subseteq\wp(\mathfrak{F}(A))$ is the topology generated by the basis $\beta=\{\phi(a):a\in A\}$ where $\phi\colon A\to\wp(\mathfrak{F}(A))$ is defined by $\phi(a)=\{x\in\mathfrak{F}(A):a\in x\}$.
    \end{enumerate}
\end{definition}

For any $I$-dimensional cylindric UV-space $X$, define: \[U\sqcup V:=\text{Int}_{\leqslant}(\text{Cl}_{\leqslant}(U\cup V)),\hspace{.3cm}\exists_{S_i}U:=S_i[U]\] for all $U,V\in\mathcal{COREG}(X)$. 
\begin{lemma}\label{functor from ba to uv}
Let $A$ be an $I$-dimensional cylindric Boolean algebra and let $X$ be an $I$-dimensional cylindric UV-space. Then: 
\begin{enumerate}
    \item $\mathcal{F}_0(A)$ is an $I$-dimensional cylindric UV-space whose specialization order is given by set inclusion; 
    \item $\mathcal{G}_0(X)=\langle\mathcal{COREG}(X);\cap,\sqcup;^*,\emptyset,X,(\exists_{S_i})_{i\in I},(\Delta_{i,k})_{i,k}\rangle$ is an $I$-dimensional cylindric Boolean algebra.
\end{enumerate}   
\end{lemma}
\begin{proof}
    For part 1, it follows by \cite[Theorem 5.4]{bezhanishvili} that $\langle\mathfrak{F}(A);\tau(\beta)\rangle$ is a UV-space whose specialization order is given by set inclusion if $A$ is a Boolean algebra. It is obvious that $S_i$ is an equivalence relation on $\mathfrak{F}(A)$ for all $i\in I$.

    To see that $S_i$ commutes with $S_k$ for all $i,k\in I$, assume that $\langle x,z\rangle\in S_i\circ S_k$ so that $\exists_i[x]=\exists_i[y]$ and $\exists_k[y]=\exists_k[z]$ for some $y\in\mathfrak{F}(A)$. Assume $a\in\exists_k[x]$ so that $a=\exists_kb$ for some $b\in x$. Then $\exists_kb\in x$ since $b\leq\exists_kb$ and $x$ is an upset. Hence $\exists_kb=a\in \exists_i[x]$ and then by our hypothesis, we have $a\in\exists_i[y]$. A similar argument shows that $a\in\exists_k[y]$ so that $\exists_k[x]\subseteq\exists_k[y]$. It can be similarly shown that $\exists_k[y]\subseteq\exists_k[x]$ as well as $\exists_i[y]=\exists_i[z]$ so that $S_i\circ S_k\subseteq S_k\circ S_i$. The proof of the $S_k\circ S_i\subseteq S_i\circ S_k$ inclusion runs the same.

We now verify that $S_k[\Delta_{i,k}\cap \Delta_{k,l}]=\Delta_{i,l}$ whenever $i,l\not=k$. First note that analogously to the proof of Lemma \ref{lemma 3.3} we have $S_k[\Delta_{i,k}\cap\Delta_{k,l}]=S_k[\phi(\delta_{i,k})\cap\phi(\delta_{k,l})]=S_k[\phi(\delta_{i,k}\wedge\delta_{k,l})]$ and hence it suffices to demonstrate $S_k[\phi(\delta_{i,k}\wedge\delta_{k,l})]=\Delta_{i,l}$. Assume that $y\in s_k[\phi(\delta_{i,k}\wedge\delta_{k,l})]$ so $xS_ky$ and hence $\exists_k[x]=\exists_k[y]$ for some $x\in\phi(\delta_{i,k}\wedge\delta_{k,l})$. Then $\delta_{i,k}\wedge\delta_{k,l}\in x$ and thus $\exists_k(\delta_{i,k}\wedge\delta_{k,l})\in x$ and hence we have $\exists_k(\delta_{i,k}\wedge\delta_{k,l})\in y$ since $\exists_k[x]=\exists_k[y]$. However $\exists_k(\delta_{i,k}\wedge\delta_{k,l})=\delta_{i,l}$ and thus $\delta_{i,l}\in y$ so $y\in\phi(\delta_{i,l})=\Delta_{i,l}$. For the converse inclusion, suppose $y\in\Delta_{i,l}$ so that $y\in\phi(\delta_{i,l})$ and hence $\delta_{i,l}\in y$. Now, conversely, suppose $y\in\Delta_{i,l}=\phi(\delta_{i,l})$. It suffices to find some $x\in\phi(\delta_{i,k}\wedge\delta_{k,l})$ such that $\exists_k[x]=\exists_k[y]$. Notice that $y\in\phi(\delta_{i,k}\wedge\delta_{k,l})$ since $\delta_{i,l}=\exists_l(\delta_{i,k}\wedge\delta_{k,l})\geq\delta_{i,k}\wedge\delta_{k,l}$ and by hypothesis we have $\delta_{i,l}\in y$ and hence we have $\exists_k[x]=\exists_k[y]$ for some $x\in\phi(\delta_{i,k}\wedge\delta_{k,l})$, i.e., when $x=y$. Therefore, we conclude $y\in S_k[\phi(\delta_{i,k}\wedge\delta_{k,l})]$, as desired.

    For condition 6, assume for the sake of contradiction that there exists some $y\in\mathfrak{F}(A)$ such that $y\in S_i[\Delta_{i,k}\cap U]\cap S_i[\Delta_{i,k}\cap U^*]$. Then $y\in S_i[\Delta_{i,k}\cap U]$ and $y\in S_i[\Delta_{i,k}\cap U]$ which implies $xS_iy$ and $zS_iy$ and thus $\exists_i[x]=\exists_i[y]$ and $\exists_i[z]=\exists_i[y]$ for some $x\in\Delta_{i,k}\cap U$ and some $z\in\Delta_{i,k}\cap U^*$. This, together with \cite[Theorem 3.13]{bezhanishvili}, implies that $x\in\phi(\delta_{i,k})\cap\phi(a)=\phi(\delta_{i,k}\wedge a)$ and $z\in\phi(\delta_{i,k})\cap\phi(a^{\perp})=\phi(\delta_{i,k}\wedge a^{\perp})$ for some $a\in A$. Therefore $\delta_{i,k}\wedge a\in x$ and $\delta_{i,k}\wedge a^{\perp}\in z$, and since $\exists_i[x]=\exists_i[y]$ and $\exists_i[z]=\exists_i[y]$, we have $\exists_i(\delta_{i,k}\wedge a)\in y$ and $\exists_i(\delta_{i,k}\wedge a^{\perp})\in y$. Since $y$ is a filter, we have $\exists_i(\delta_{i,k}\wedge a)\wedge\exists_i(\delta_{i,k}\wedge a^{\perp})\in y$ but $\exists_i(\delta_{i,k}\wedge a)\wedge\exists_i(\delta_{i,k}\wedge a^{\perp})=0$ and so $0\in y$, which contradicts the fact that $y$ is proper. Condition 8 follows from the definition of $S_i$ in $\mathcal{F}_0(A)$ and the fact that $\mathcal{COREG}_{\mathcal{F}_0(A)}(x)=\{\phi(c)\in\mathcal{COREG}(\mathcal{F}_0(A)):x\in\phi(c)\}$. Conditions 7 and 9 follow from Lemma \ref{functor from ba to uv} and \cite[Theorem 3.13]{bezhanishvili}.  

    For part 2, it follows by \cite[Proposition 5.3]{bezhanishvili} that $\langle\mathcal{COREG}(X);\cap,\sqcup,^*,\emptyset,X\rangle$ is a Boolean algebra whenever $X$ is a UV-space. The proof that $\exists_{S_i}$ is a quantifier on $\langle\mathcal{COREG}(X);\cap,\sqcup,^*,\emptyset,X\rangle$ follows analogously to \cite{halmos1}. Moreover, it is obvious that $\exists_{S_i}\exists_{S_k}U=\exists_{S_k}\exists_{S_i}U$ for all $i,k\in I$ and each $U\in\mathcal{COREG}(X)$ since $S_i$ and $S_k$ are commuting relations. Conditions 4-8 of Definition \ref{cylindric uv space} guarantee that $\Delta_{i,k}$ is a diagonal element in $\langle\mathcal{COREG}(X);\cap,\sqcup,^*,\emptyset,X,(\exists_{S_i})_{i\in I}\rangle$ for all $i,k\in I$.     
\end{proof}

The following lemma describes some useful properties of cylindric UV-spaces. 
\begin{lemma}\label{cylindric uv lemma}
    If $X$ is an $I$-dimensional cylindric UV-space, then: 
    \begin{enumerate}
     \item the specialization order $\leqslant$ is a partial order;
    \item $x\not\leqslant y$ implies there exists $U\in\mathcal{COREG}(X)$ such that $x\in U$ and $y\not\in U$;
        \item $xS_iy$ implies $\exists_{S_i}[\mathcal{COREG}_X(x)]=\exists_{S_i}[\mathcal{COREG}_X(y)]$.
    \end{enumerate}
\end{lemma}
\begin{proof}
    Parts 1 and 2 follow analogously to conditions 1 and 2 of Lemma \ref{muvo lemma}. For part 3, assume that $xS_iy$ for all $i\in I$ and let $U\in\exists_{S_i}[\mathcal{COREG}_X(x)]$. Then $U=\exists_{S_i}V$ for some $V\in\mathcal{COREG}_X(x)$ and hence $x\in V$. Since: 
    \[\exists_{S_i}V=S_i[V]=\{z\in X:wS_iz\hspace{.2cm}\text{for some}\hspace{.2cm}w\in V\}\] it follows that $y\in S_i[V]=\exists_{S_i}V$. Therefore $\exists_{S_i}V\in\mathcal{COREG}_X(y)$ and since we have assumed that $U=\exists_{S_i}V$, we have $U\in\exists_{S_i}[\mathcal{COREG}_X(y)]$. Thus we conclude that $\exists_{S_i}[\mathcal{COREG}_X(x)]\subseteq\exists_{S_i}[\mathcal{COREG}_X(y)]$. The $\exists_{S_i}[\mathcal{COREG}_X(y)]\subseteq\exists_{S_i}[\mathcal{COREG}_X(x)]$ inclusions runs similarly by the symmetry of $S_i$.       
\end{proof}
We now arrive at the topological representation theorem for $I$-dimensional cylindric Boolean algebras and the algebraic realization theorem for $I$-dimensional cylindric UV-spaces. 
\begin{theorem}\label{representation of ba}
    Every $I$-dimensional cylindric Boolean algebra $A$ is isomorphic to $\mathcal{G}_0(\mathcal{F}_0(A))$. Moreover, every $I$-dimensional cylindric UV-space $X$ is relationally homeomorphic to $\mathcal{F}_0(\mathcal{G}_0(X))$. 
\end{theorem}
\begin{proof}
    By \cite[Theorem 3.13]{bezhanishvili} the Boolean algebra reduct of $A$ is isomorphic to the Boolean algebra reduct of $\mathcal{G}_0(\mathcal{F}_0(X))$ under $\phi(a)=\{x\in\mathfrak{F}(A):a\in x\}$. Also $\phi(\delta_{i,k})=\Delta_{i,k}$ and hence it remains to see that $\exists_i\mapsto\exists_{S_i}$ is a homomorphic mapping under $\phi$. It suffices to show that $\phi(\exists_ia)=\exists_{S_i}\phi(a)$. Since we have $\phi(a)=\{x\in\mathfrak{F}(A):a\in x\}$, it follows that $S_i[\phi(a)]=\{x\in\mathfrak{F}(A):\exists_ia\in x\}$ and hence we conclude $\phi(\exists_ia)=S_i[\phi(a)]=\exists_{S_i}\phi(a)$.
    
    By \cite[Theorem 5.4]{bezhanishvili}, it follows that $X$ is homeomorphic to $\mathcal{F}_0(\mathcal{G}_0(X))$ under $\psi(x)=\{U\in\mathcal{COREG}(X):x\in U\}$. To see that $\psi$ is a relational isomorphism, i.e., that $xS_i$ iff $\psi(x)S_i\psi(y)$ for all $i\in I$, we begin with the left-to-right implication by assuming that $xS_iy$. By Lemma \ref{cylindric uv lemma}, we have $\exists_{S_i}[\mathcal{COREG}_X(x)]=\exists_{S_i}[\mathcal{COREG}_X(y)]$, which by the definition of $\psi$ implies $\exists_{S_i}[\psi(x)]=\exists_{S_i}[\psi(s)]$ and therefore $\psi(x)S_i\psi(y)$. The converse implication follows by a contrapositive proof using condition 8 of Definition \ref{cylindric uv space}.    
\end{proof}

We now give a topological construction of the canonical completion of a cylindric Boolean algebra. 

\begin{theorem}\label{completion of ba}
    Let $A$ is an $I$-dimensional cylindric Boolean algebra and let $X$ be its dual $I$-dimensional cylindric UV-space. Then the completion $\langle\phi;\mathcal{REG}(X)\rangle$ is its canonical completion. 
\end{theorem}
\begin{proof}
    By Tarski \cite{tarski1, tarski2}, it is known that $\mathcal{REG}(X)$ forms a Boolean algebra for any topological space $X$. By \cite[Theorem 8.27]{bezhanishvili}, the function $\phi$ is a dense and compact homomorphic embedding from the Boolean algebra reduct of $A$ into the Boolean algebra reduct of $\mathcal{REG}(X)$. Moreover, it can be shown similarly to Lemma \ref{functor from ba to uv}, that $\mathcal{REG}(X)$ forms a complete $I$-dimensional cylindric Boolean algebra. Lastly, by Theorem \ref{representation of ba}, it follows that $\phi$ is an $I$-dimensional cylindric Boolean algebra homomorphism.     
\end{proof}
The following are relational extensions of the UV-maps introduced by Bezhanishvili and Holliday \cite{bezhanishvili}. 
\begin{definition}\label{uv map}
    Let $X$ and $X'$ be $I$-dimensional cylindric UV-spaces. A function $f\colon X\to X'$ is a \emph{cylindric UV-map} if the following conditions are satisfied: 
    \begin{enumerate}
        \item $f$ is a spectral map; 
        \item if $f(x)\leqslant' y'$, there exists $y\in X$ such that $x\leqslant y$ and $f(y)=y'$; 
        \item $S_{i}[f^{-1}[U]]=f^{-1}[S_{i}[U]]$ for all $U\in\mathcal{COREG}(X)$; 
   \item $f^{-1}[\Delta'_{i,k}]=\Delta_{i,k}$. 
    \end{enumerate}
\end{definition}
If $X$ and $X'$ are UV-spaces and $f\colon X\to X'$ is a function, then $f$ is a \emph{UV-map} provided conditions 1 and 2 in Definition \ref{uv map} are satisfied. 
 \begin{proposition}[\protect{\cite[Corollary 6.4]{bezhanishvili}}]\label{preimage of reg is reg}
    Let $X$ and $X'$ be UV-spaces and let $f\colon X\to X'$ be a UV-map. Then $f^{-1}[U]\in\mathcal{COREG}(X)$ for each $U\in\mathcal{COREG}(X')$. 
\end{proposition}

By $\mathcal{CBA}$ we denote the category of $I$-dimensional cylindric Boolean algebras and homomorphisms and by $\mathcal{CUV}$ we denote the category of $I$-dimensional cylindric UV-spaces and UV-maps. 

\begin{theorem}\label{ba homo to uv map}
     Let $X$ and $X'$ be $I$-dimensional cylindric UV-spaces and let $f\colon X\to X'$ be an $I$-dimensional cylindric UV-map. Then $\mathcal{G}_1(f)\colon \mathcal{G}_0(X')\to\mathcal{G}_0(X)$ defined by $\mathcal{G}_1(f):=f^{-1}$ is an $I$-dimensional cylindric Boolean algebra homomorphism. Therefore $\mathcal{G}_F=\langle\mathcal{G}_0,\mathcal{G}_1\rangle\colon\mathcal{CUV}\to\mathcal{CBA}$ determines a contravariant functor. Moreover, let $A$ and $A'$ be $I$-dimensional cylindric Boolean algebras and let $h\colon A\to A'$ be an $I$-dimensional cylindric Boolean algebra homomorphism. Then the map $\mathcal{F}_1(h)\colon\mathcal{F}_0(A')\to\mathcal{F}_0(A)$ defined by $\mathcal{F}_1(h):=h^{-1}$ is an $I$-dimensional cylindric UV-map. Therefore $\mathcal{F}_F=\langle\mathcal{F}_0,\mathcal{F}_1\rangle\colon\mathcal{CBA}\to\mathcal{CUV}$ determines a contravariant functor.  
\end{theorem}
\begin{proof}
By \cite[Theorem 6.7]{bezhanishvili}, it follows that if $h\colon A\to A'$ is a homomorphism of Boolean algebras, then $\mathcal{F}_1(h)=h^{-1}\colon\mathcal{F}_0(A')\to\mathcal{F}_0(A)$ is a UV-map and that if $f\colon X\to X'$ is a UV-map between UV-spaces, then $\mathcal{G}_1(f)=f^{-1}\colon\mathcal{G}_0(X')\to\mathcal{G}_0(X)$ is a Boolean algebra homomorphism. The remainder of the part 1 follows similarly to Theorem \ref{continuous to homomorphism} and part 2 follows similarly to Theorem \ref{functors}. 
\end{proof}

\begin{theorem}\label{duality for ba}
$\mathcal{G}_F=\langle\mathcal{G}_0,\mathcal{G}_1\rangle\colon\mathcal{CUV}\to\mathcal{CBA}$ and $\mathcal{F}_F=\langle\mathcal{F}_0,\mathcal{F}_1\rangle\colon\mathcal{CBA}\to\mathcal{CUV}$ constitute a dual equivalence between the categories $\mathcal{CBA}$ and $\mathcal{CUV}$.  
\end{theorem}
It has already been pointed out that ortholattices generalize Boolean algebras in the sense that an ortholattice $A$ is a Boolean algebra if and only if $A$ is distributive. The following observation made by Bezhanishvili and Holliday in \cite{bezhanishvili} therefore connects the duality presented in this section for cylindric Boolean algebras with the duality obtained in Section 6 for cylindric ortholattices. 

\begin{proposition}[\protect{\cite[pg.~146]{bezhanishvili}}]\label{reg is biortho}
    If $A$ is a distributive ortholattice (i.e., a Boolean algebra), then $\mathcal{B}(\mathfrak{F}(A))=\mathcal{REG}(\mathfrak{F}(A))$ where $\mathcal{REG}(\mathfrak{F}(A))$ is defined with respect to filter inclusion. 
\end{proposition}

Indeed, by Theorem \ref{topological representation} every cylindric ortholattice $A$ is isomorphic to $\mathcal{COB}(X)$ where $X$ is the cylindric UVO-space induced by the spectrum $\mathfrak{F}(A)$ of proper filters of $A$, and by Theorem \ref{representation of ba}, every cylindric Boolean algebra is isomorphic to $\mathcal{COREG}(X)$ where $X$ is the UV-space induced to the spectrum $\mathfrak{F}(A)$ of proper filters of $A$. Moreover, by Theorem \ref{algebraic realization}, every UVO-space $X$ is relationally homeomorphic to $\mathfrak{F}(A)$ where $A$ is the cylindric ortholattice $\mathcal{COB}(X)$ of compact open biorthogonally closed subsets of $X$, and every cylindric UV-space $X$ is relationally homeomorphic to $\mathfrak{F}(A)$ where $A$ is the cylindric Boolean algebra $\mathcal{COREG}(X)$ of compact open regular open subsets of $X$. Therefore, Proposition \ref{reg is biortho} describes how the duality obtained for cylindric Boolean algebras is related to the duality obtained for cylindric ortholattices.

\section{Conclusion and future work}

 We have extended the canonical completion results obtained in \cite{harding3} to the setting of cylindric ortholattices. We then extended the duality results obtained in \cite{mcdonald} and \cite{harding3} to a choice-free duality for cylindric ortholattices. Then, by building on the results obtained in \cite{holliday, bezhanishvili}, we provided a topological construction of the canonical completion of cylindric Boolean algebras as well as a duality result for cylindric Boolean algebras. 

 Related lines of research may include investigating the MacNeille completion of cylindric ortholattices. This will involve extending the results concerning the MacNeille completion of a monadic ortholattice obtained in \cite{harding3}.

\end{document}